\documentclass[12pt,a4paper,reqno]{amsart}
\usepackage{vmargin,color}
\usepackage[latin1]{inputenc}
\usepackage{amsmath,amsfonts,amsthm,epsfig}

\def\R{\mathbb R}
\def\N{\mathbb N}

\def\Rp{\R^{+}}
\def\Sn{S^{n-1}}

\def\cal{\mathcal}

\def\F{{\cal F}}

\def\a{\alpha}
\def\b{\beta}

\def\de{\delta}
\def\e{\varepsilon}
\def\k{\kappa}
\def\l{\lambda}

\def\s{\sigma}
\def\vphi{\varphi}

\newtheorem{theorem}{Theorem}

\newtheorem{corollary}[theorem]{Corollary}

\begin{document}

\title{Quantitative stability estimate for an optimization problem under constraints}

\author{H. Hajaiej}

\begin{abstract} A class of functionals maximized by characteristic functions of balls is identified by a mass transportation argument.
\end{abstract}

\maketitle

A variational approach to the study of standing waves for the nonlinear Schr\"odinger equation leads to the minimization of functionals like
\begin{equation}\label{optic functional}
\frac{1}2\int_{\R^n}|\nabla u(x)|^2\,dx-\int_{\R^n}F(|x|,u(x))\,dx\,,
\end{equation}
over all $u\in H^1(\R^n)$, $u\ge0$, such that $\int_{\R^n}u^2=1$, see \cite{HS} (here $n\in\N, n\ge1$). The function $F$ describes the index of refraction of the media in which the wave propagates. A typical example is
\[
F(r,s)=p(r)s^2+q(r)s^{d}\,,\quad 2<d<2+\frac4{n}\,,
\]
where $p$ and $q$ are positive decreasing functions, and the constraint on $d$ has to be assumed so to avoid non-existence issues due to unbalanced scalings. The two terms of the energy \eqref{optic functional} are in competition. Indeed, if we try to maximize $\int_{\R^n}F(|x|,u(x))dx$ under the additional constraint that $u\le a$, then the unique maximizer is given by the function $a\, 1_{r B}$, having infinite Dirichlet integral (here $B$ is the Euclidean unit ball and $r>0$ is such that $\int_{\R^n}(a\,1_{rB})^2=1$).

In this note we identify a simple sufficient condition on an integrand $F$ ensuring that $\int_{\R^n}F(|x|,u(x))dx$ presents this behavior. More precisely, we are going to consider integrands $F:\Rp\times\Rp\to\Rp$ (here $\Rp:=[0,\infty)$), such that
\begin{enumerate}
\item[(H1)] for every $s\in\Rp$, $F(\cdot,s)$ is decreasing; for a.e. $r\in\Rp$, $F(r,\cdot)$ is continuous on $\Rp$;
\item[(H2)] there exist $\a\in L^1(\Rp,r^{n-1}dr)$ and a locally bounded function $\b:\R\to\Rp$ such that, for a.e. $r\in\Rp$ and every $s\in\Rp$,  $F(r,s)\le \a(r)\b(s)$.
\end{enumerate}
Given $a>0$ and $p\ge 1$, we consider the convex subset of $L^p(\R^n)$
\[
X:=\left\{u\in L^p(\R^n):0\le u\le a\,, \int_{\R^n}u^p\le 1\right\}\,,
\]
and define a functional $\F$ on $X$ by setting
\begin{gather*}
\F(u)=\int_{\R^n}F(|x|,u(x))dx\,,\quad\forall u\in X\,.
\end{gather*}
Note that, thanks to (H1) and (H2), $x\in\R^n\mapsto F(|x|,u(x))$ is measurable and $\F(u)\in\Rp$ for every $u\in X$. We are going to prove the following theorem:

\begin{theorem}\label{thm: main} Let $a>0$, $p\ge 1$, and let $F$ be such that (H1) and (H2) hold true. Assume that there exists $t>0$ such that the ball $E=\{x\in\R^n:F(|x|,a)>t\}$ satisfies $a^p|E|=1$, and that, for a.e. $r\in\Rp$ and for every $\l\in[0,1]$,
\begin{equation}\label{condition}
F(r,\l a)\le \l^p F(r,a)\,.
\end{equation}
Then the function $w=a\,1_E$ is a maximum of $\F$ on $X$. Moreover, if $F(\cdot,a)$ is strictly decreasing, then $w=a\,1_E$ is the unique maximizer of $\F$ on $X$.
\end{theorem}

The proof of Theorem \ref{thm: main} is based on a basic result in mass transportation theory, namely the Brenier Theorem \cite{Brenier} (see also \cite{McCann1}): given two Radon measures $\mu_1,\mu_2$ on $\R^n$, both absolutely continuous with respect to the Lebesgue measure and such that $\mu_1(\R^n)=\mu_2(\R^n)$, there exists a convex function $\vphi:\R^n\to[0,\infty]$ and a Borel measurable map $T:\R^n\to\R^n$ such that $T(x)=\nabla\vphi(x)$ at a.e. $x\in\R^n$ and $T$ pushes forward $\mu_1$ into $\mu_2$, i.e.
\begin{equation}\label{push-forward general}
\int_{\R^n}H(y)d\mu_2(y)=\int_{\R^n}H(T(x))d\mu_2(x)\,,
\end{equation}
for every Borel function $H:\R^n\to[0,\infty]$. The mass transportation approach to Theorem \ref{thm: main} allows also to deduce a quantitative stability estimate on the maximality of $w=a1_E$, see Corollary \ref{cor: quant} below. We pass now to prove Theorem \ref{thm: main}.


\begin{proof}[Proof of Theorem \ref{thm: main}] By \eqref{condition}, as $F(r,\cdot)$ is continuous for a.e. $r\in\Rp$, we deduce that $F(r,0)=0$ for a.e. $r\in\Rp$. We let $\Sn=\{x\in\R^n:|x|=1\}$, and denote by $\s$ the $(n-1)$-dimensional Hausdorff measure restricted to $\Sn$.

{\it Step one:} Let us fix $u\in X$ and construct an auxiliary function $v=a\,1_G$ by letting
\[
G:=\left\{x\in\R^n: |x|<\k\left(\frac{x}{|x|}\right)\right\}\,,
\]
where we have introduced $\k:\Sn\to\Rp$,
\begin{equation}\label{def of knu}
\k(\nu):=\left(\frac{n}{a^p}\int_0^\infty u(r\nu)^pr^{n-1}dr\right)^{1/n}\,,\quad \nu\in\Sn\,.
\end{equation}
Note that $v(r\nu)=a\,1_{[0,\k(\nu)]}(r)$, and that the value of $\k(\nu)$ has been chosen so that the measures
\[
1_{\Rp}(r)u(r\nu)^p\,r^{n-1}\,dr\,\quad\mbox{and}\quad 1_{\Rp}(r)v(r\nu)^p\,r^{n-1}\,dr\,,
\]
have the same total mass on $\R$. For every $\nu\in\Sn$, let $T_\nu$ denote the map given by Brenier theorem. By construction $T_\nu$ is increasing on $\R$, moreover, thanks to \eqref{push-forward general} we have
\begin{equation}\label{pushforwardTnu}
\int_{\Rp}H(r)v(\nu r)^pr^{n-1}dr=\int_{\Rp}H(T_\nu(r))u(\nu r)^pr^{n-1}dr\,,
\end{equation}
for every Borel function $H:\R\to[0,\infty]$: in particular $T_\nu(r)\in[0,k(\nu)]$ for a.e. $r\in\R$. Note also that, as $0\le u\le a$, we clearly have
\begin{equation}\label{Tnu sotto r}
T_\nu(r)\le r\,,\quad \mbox{for a.e. $r\in\Rp$.}
\end{equation}
We are going to prove that $\F(u)\le\F(v)$. By \eqref{condition} we have that
\begin{eqnarray}
\F(u)=\int d\s(\nu)\int_{\Rp}F(r,u(r\nu))r^{n-1}dr\le \label{u<v}
\int d\s(\nu)\int_{\Rp}\frac{F(r,a)}{a^p}u(r\nu)^{p}r^{n-1}dr\,,
\end{eqnarray}
while at the same time, thanks to \eqref{pushforwardTnu}
\begin{eqnarray*}
\F(v)&=&\int d\s(\nu)\int_{0}^{\k(\nu)}F(r,a)r^{n-1}dr
=\int d\s(\nu)\int_{\Rp}\frac{F(r,a)}{a^p}v(r\nu)^pr^{n-1}dr
\\
&=&\int d\s(\nu)\int_{\Rp}\frac{F(T_\nu(r),a)}{a^p}u(r\nu)^pr^{n-1}dr\,,
\end{eqnarray*}
By (H1) and \eqref{Tnu sotto r} it follows immediately that $\F(u)\le\F(v)$.

\medskip

{\it Step two:} We are going to prove that $\F(v)\le\F(w)$. We start by noticing that $|E|=|G|$. Indeed by \eqref{def of knu}
\[
|G|=\int \frac{\k(\nu)^n}n\,d\s(\nu)=\frac1{a^p}\int_{\R^n}u^p=|E|\,.
\]
In particular $|E\setminus G|=|G\setminus E|$, and, without loss of generality, $|E\setminus G|>0$. Consider the Brenier map $T:\R^n\to\R^n$ between $1_{E\setminus G}(x)dx$ and $1_{G\setminus E}(y)dy$. By \eqref{push-forward general},
\begin{equation}\label{push-forward T}
\int_{E\setminus G}H(y)dy=\int_{G\setminus E}H(T(x))dx\,,
\end{equation}
for every Borel function $H:\R^n\to[0,\infty]$. On choosing $H(y)=F(y,a)$ we find
\begin{eqnarray}\label{step3}
\int_{E\setminus G}F(|x|,a)dx
=\int_{G\setminus E}F(|T(x)|,a)dx\,,
\end{eqnarray}
while, on taking $H(y)=1_{E\setminus G}(y)$, we prove that $T(x)\in E\setminus G$ for a.e. $x\in G\setminus E$. As $E$ is a ball, this last remark implies that
\begin{eqnarray}\label{T sotto |x|}
|T(x)|\le |x|\,,\quad\mbox{for a.e. $x\in G\setminus E$\,.}
\end{eqnarray}
On combining \eqref{T sotto |x|} with \eqref{step3} we get
\begin{eqnarray}\nonumber
\F(w)&=&\int_{G\cap E}F(|x|,a)dx+\int_{E\setminus G}F(|x|,a)dx
\\\nonumber
&=&\int_{G\cap E}F(|x|,a)dx+\int_{G\setminus E}F(|T(x)|,a)dx
\\\label{v<w}
&\ge&\int_{G\cap E}F(|x|,a)dx+\int_{G\setminus E}F(|x|,a)dx=\F(v)\,,
\end{eqnarray}
and the conclusion follows.

Let us now assume that for every $s\in\Rp$ the function $F(\cdot,a)$ is strictly decreasing, and consider a function $u\in X$ that maximizes $\F$ on $X$, i.e. such that $\F(u)=\F(w)$. We want to show that $u=w$ a.e. on $\R^n$. Let us prove that $G=E$ up to null sets. Indeed, let $R$ denote the radius of the ball $E$. If $|G\setminus E|>0$, then we can consider $T$ and repeat the above argument. Since $F(\cdot,a)$ is strictly decreasing and equality holds in \eqref{v<w}, we find that $|T(x)|=|x|$ for a.e. $x\in G\setminus E$. Thus $|T(x)|\ge R$ for a.e. $x\in\R$; but $T(x)\in E\setminus G$ for a.e. $x\in G\setminus E$, therefore it must be $|G\setminus E|=0$, a contradiction. As $G=E$ up to null sets, we have $\k(\nu)=R$ for every $\nu\in\Sn$. The equality sign in \eqref{u<v} implies that, for $\s$-a.e. $\nu\in\Sn$, $T_\nu(r)=r$ for a.e. $r\in\{t:u(\nu t)>0\}$. As $0\le T_\nu\le\k(\nu)=R$, by \eqref{pushforwardTnu} and \eqref{Tnu sotto r} we deduce that $\{t:u(\nu t)>0\}\subset[0,R]$ for $\s$-a.e. $\nu\in\Sn$. On applying \eqref{pushforwardTnu} to $H=1_{\{t:u(\nu t)>0\}}$ we deduce $u(\nu r)=a$ on $\{t:u(\nu t)>0\}$, therefore that $u(\nu r)=a\,1_{[0,R]}(r)$. In particular $u=w$ a.e. on $\R^n$.
\end{proof}

We come now to a quantitative stability estimate:

\begin{corollary}\label{cor: quant} Under the assumptions of Theorem \ref{thm: main}, let us assume the existence of $\l>0$ such that, whenever $0<r_1<r_2$,
\begin{equation}\label{l}
F(r_1,a)\ge F(r_2,a)+\l(r_2-r_1)\,.
\end{equation}
Then, for every $u\in X$ we have that
\begin{equation}\label{tesi cor}
\int_{\R^n}|u-w|^p\le C(n,p,a)\sqrt{\frac{\F(w)-\F(u)}\l}\,.
\end{equation}
where $C(n,p,a)$ is a constant depending only on $n$, $p$ and $a$.
\end{corollary}

\begin{proof} Let $\delta:=\F(w)-\F(u)$. Thanks to \eqref{l}, from \eqref{u<v} and \eqref{v<w} we find that
\begin{eqnarray}\label{quant1}
\de&\ge&\l\int_{G\setminus E} (|x|-|T(x)|)dx\,,
\\ \label{quant2}
\de&\ge&\l\int d\s(\nu)\int_0^\infty (r-T_\nu(r))\frac{u(r\nu)^p}{a^p}r^{n-1}dr\,.
\end{eqnarray}
We now consider \eqref{quant1} and \eqref{quant2} separately:

{\it Step one:} Let $\e\in(0,R)$, then $(R+\e)^n\le R^n+c(n)R^{n-1}\e$. Thus
\begin{eqnarray*}
\frac1{a^p}\int_{\R^n}|w-v|^p&=&|E\Delta G|=2|G\setminus E|
\le2\{|G\setminus B_{R+\e}|+|B_{R+\e}\setminus B_R|\}
\\
&\le& C(n)\left\{|G\setminus B_{R+\e}|+\e R^{n-1}\right\}\,.
\end{eqnarray*}
If $x\in G\setminus B_{R+\e}$ then $|x|\ge R+\e\ge|T(x)|+\e$. By \eqref{quant1} we have $|G\setminus B_{R+\e}|\le (\de/\e\l)$, therefore we come to
\[
\frac1{a^p}\int_{\R^n}|w-v|^p\le C(n)\left\{\frac{\de}{\l\e}+\e R^{n-1}\right\}\,.
\]
We minimize over $\e\in[0,R]$ and find
\begin{equation}\label{end of step 1}
\frac1{a^p}\int_{\R^n}|w-v|^p\le C(n)\max\left\{\sqrt{\frac{\de R^{n-1}}{\l}},\frac{\de}{\l R}\right\}\,.
\end{equation}

{\it Step two:} We start by noticing that
\begin{eqnarray*}
\int_{\R^n}|u-v|^p&=&\int_G (a-u)^p+\int_{\R^n\setminus G}u^p
\le \int_G (a^p-u^p)+\int_{\R^n\setminus G}u^p=2\int_{\R^n\setminus G}u^p
\\&=&2\int\tau_2(\nu)d\s(\nu)\,,
\end{eqnarray*}
where, for every $\nu\in\Sn$, we have set
\[
\tau_1(\nu):=\int_0^{\k(\nu)}u(r\nu)^pr^{n-1}dr\,,\quad
\tau_2(\nu):=\int_{\k(\nu)}^\infty u(r\nu)^pr^{n-1}dr\,.
\]
Since $a^p\k(\nu)^n/n=\tau_1(\nu)+\tau_2(\nu)$, we have that
\begin{eqnarray*}
a^p\frac{T_\nu(r)^n}n&\le&\tau_1(\nu)+\int_{\k(\nu)}^r u(t\nu)t^{n-1}dt\le\tau_1(\nu)+a^p\frac{r^n}{n}-a^p\frac{\k(\nu)^n}n\,,
\end{eqnarray*}
for every $r\ge\k(\nu)$, i.e.,
\[
T_\nu(r)\le\left(r^n-\frac{n\tau_2(\nu)}{a^p}\right)^{1/n}\,,\quad\forall r\ge \k(\nu)\,.
\]
Then by \eqref{quant2} we deduce that
\begin{eqnarray*}
\frac{\de}{\l}&\ge&\int d\s(\nu)\int_{\k(\nu)}^\infty \left[r-\left(r^n-\frac{n\tau_2(\nu)}{a^p}\right)^{1/n}\right]u(r\nu)^pr^{n-1}dr
\\
&\ge&\int d\s(\nu)\int_{\k(\nu)}^\infty \frac{\tau_2(\nu)}{a^p}\left(r^n-\frac{n\tau_2(\nu)}{a^p}\right)^{(1/n)-1}u(r\nu)^pr^{n-1}dr
\\
&\ge&\int \frac{\tau_2(\nu)^2}{a^p}\left(\k(\nu)^n-\frac{n\tau_2(\nu)}{a^p}\right)^{(1/n)-1}d\s(\nu)
\\
&=&\frac{c(n)}{a^{p/n}}\int\tau_2(\nu)^2\tau_1(\nu)^{(1/n)-1}d\s(\nu)\,.
\end{eqnarray*}
By H\"older inequality
\begin{eqnarray*}
\int_{\R^n}|u-v|^p\le 2\int\tau_2(\nu)d\s(\nu)\le C(n) \sqrt{a^{p/n}\frac\de\l}\sqrt{\int \tau_1(\nu)^{1-(1/n)}d\s(\nu)}\,.
\end{eqnarray*}
By Jensen inequality for concave functions,
\begin{eqnarray*}
\int \tau_1(\nu)^{1-(1/n)}d\s(\nu)&\le& C(n)\left(\int \tau_1(\nu)d\s(\nu)\right)^{1-(1/n)}
\\&\le& C(n)\left(\int a^p\frac{\k(\nu)^n}{n}d\s(\nu)\right)^{1-(1/n)}= C(n)\,,
\end{eqnarray*}
and we come to conclude that
\begin{equation}\label{end of step 2}
\int_{\R^n}|u-v|^p\le C(n)\sqrt{a^{p/n}\frac{\de}{\l}}
\end{equation}

{\it Step three:} As $\int |w-u|^p\le 2^p$, \eqref{tesi cor} follows trivially whenever $\de\ge\le$. Let us now assume that $\de\le\l$, then \eqref{tesi cor} is easily deduced from \eqref{end of step 1} and \eqref{end of step 2}.
\end{proof}

\end{document}